\documentclass[11pt]{amsart}

\usepackage{epigamath}
\usepackage{microtype}

\usepackage[english]{babel}

\numberwithin{equation}{section}

\usepackage{mathdots}
\usepackage{color}
\usepackage{hyperref,url}
\usepackage{tikz, tikz-cd}
\usepackage{mathtools}
\usepackage{stmaryrd}

\newtheorem{theorem}{Theorem}[section]
\newtheorem{corollary}[theorem]{Corollary}
\newtheorem{lemma}[theorem]{Lemma}
\newtheorem{proposition}[theorem]{Proposition}
\newtheorem{conjecture}[theorem]{Conjecture}

\theoremstyle{definition}

\newtheorem{remark}[theorem]{Remark}
\newtheorem{question}[theorem]{Question}
\newtheorem{speculation}[theorem]{Speculation}

\newcommand{\ul}[1]{\mathbf{#1}}

\def\oM{\overline{\mathcal{M}}}
\def\cM{{\mathcal{M}}}

\def\qed{{\hfill $\square$}}
\def\b1{{\bf 1}}

\EpigaVolumeYear{4}{2020} \EpigaArticleNr{12} \ReceivedOn{June 26,
2019}
\InFinalFormOn{June 3, 2020}
\AcceptedOn{July 23, 2020}

\title{Zero cycles on the moduli space of curves}
\titlemark{Zero cycles on the moduli space of curves}

\author{Rahul Pandharipande and Johannes Schmitt}
\address{Departement Mathematik, ETH Z\"urich}
\email{\texttt{rahul@math.ethz.ch} and \texttt{johannes.schmitt@math.ethz.ch}}

\authormark{R. Pandharipande and J. Schmitt}

\AbstractInEnglish{\scriptsize{While the Chow groups of $0$-dimensional cycles on the
  moduli spaces of Deligne-Mumford stable pointed curves can be
  very complicated, the span of the $0$-dimensional tautological cycles 
  is always of rank 1. The question of whether a given
  moduli point $[C,p_1,\ldots,p_n] \in \oM_{g,n}$ determines a tautological
  $0$-cycle is subtle. Our main results address the question for
  curves on rational and $K3$ surfaces. If $C$ is a nonsingular
  curve on a nonsingular rational surface of positive degree
  with respect to the anticanonical class, we prove
  $[C,p_1,\ldots,p_n]$ is tautological if
  the number of markings does not exceed the virtual dimension
  in Gromov-Witten theory of the moduli space of stable maps.
   If $C$ is a nonsingular
  curve on a $K3$ surface, we prove
  $[C,p_1,\ldots,p_n]$ is tautological if 
  the number of markings does not exceed the genus of $C$ {\em and}
every marking is a Beauville-Voisin point. 
  The latter result provides a connection between the rank 1 tautological 
  $0$-cycles on the moduli of curves and the rank 1 tautological $0$-cycles on
  $K3$ surfaces.

  Several further results related to tautological $0$-cycles
on the moduli spaces of curves are proven.
  Many open questions
  concerning the moduli points of curves on other surfaces (Abelian, Enriques, general
  type) are discussed.}}



\TitleInFrench{Z\'ero cycles sur l'espace de modules des courbes}

\AbstractInFrench{\scriptsize{Alors que les groupes de Chow des z\'ero-cycles sur les espaces de modules de Deligne-Mumford des courbes stables point\'es peuvent \^etre tr\`es compliqu\'es, le sous-groupe des z\'ero-cycles tautologiques est toujours de rang $1$. Savoir si un point de l'espace de modules $[C,p_1,\ldots,p_n] \in \oM_{g,n}$ d\'etermine un z\'ero-cycle tautologique est une question subtile. Nos r\'esultats principaux r\'epondent \`a cette question pour les courbes sur les surfaces rationnelles et les surfaces $K3$. Si $C$ est une courbe lisse sur une surface rationnelle lisse, de degr\'e anticanonique positif, nous d\'emontrons que $[C,p_1,\ldots,p_n]$ est tautologique si le nombre de points marqu\'es n'exc\`ede pas la dimension virtuelle de l'espace de modules des applications stables en th\'eorie de Gromov-Witten. Si $C$ est une courbe lisse sur une surface $K3$, $[C,p_1,\ldots,p_n]$ est tautologique si le nombre de points marqu\'es n'exc\`ede pas le genre de $C$ \emph{et} si tout marquage est un point de Beauville-Voisin. Ce dernier r\'esultat fournit une connexion entre les z\'ero-cycles sur l'espace de modules des courbes et les z\'ero-cycles tautologiques sur les surfaces $K3$.

Plusieurs autres r\'esultats reli\'es aux z\'eros-cycles tautologiques sur les espaces de modules de courbes sont \'etablis et nous discutons de nombreuses questions ouvertes concernant les points correspondants aux courbes sur d'autres surfaces (ab\'eliennes, d'Enriques, de type g\'en\'eral) font l'objet de discussion.}}

\acknowledgement{R.P. was supported by SNF-200020-182181, 
ERC-2017-AdG-786580-MACI, SwissMAP, and the Einstein Stiftung.
J.S. was supported by SNF-200020-162928 and ERC-2017-AdG-786580-MACI, by the SNF early postdoc mobility grant 184245 and wants to thank
the Max Planck Institute for Mathematics in Bonn for its hospitality.
This project has received funding from the European Research Council (ERC)
under the European Union Horizon 2020 research and innovation programme
(grant agreement No 786580).}

\begin{document}


\removeabove{20pt}
\removebetween{20pt}
\removebelow{20pt}

\maketitle

\begin{prelims}

\DisplayAbstractInEnglish

\medskip

\DisplayKeyWords

\medskip

\DisplayMSCclass

\medskip

\languagesection{Fran\c{c}ais}

\medskip

\DisplayTitleInFrench

\medskip

\DisplayAbstractInFrench

\end{prelims}


\newpage

\setcounter{tocdepth}{2}

\tableofcontents


\section{Introduction}

\subsection{Moduli of curves} \label{xcxc}

Let $(C,p_1,\ldots,p_n)$ be a Deligne-Mumford stable curve of
genus $g$ with $n$ marked points
defined over $\mathbb{C}$.
Let  
$$[C,p_1,\ldots,p_n]\in  {\oM}_{g,n}\,  $$
be the associated {\em moduli point} in 
the moduli space.{\footnote{Stability requires $2g-2+n>0$ which
we always impose when we write $\oM_{g,n}$.}}
As a Deligne-Mumford stack, ${\oM}_{g,n}$
is nonsingular, irreducible, and of complex dimension $3g-3+n$.
Though the moduli spaces ${\oM}_{g,n}$ can be irrational and complicated, their study has been
marked by the discovery of beautiful mathematical structures.
 
Fundamental to the geometry of the 
 moduli spaces of stable pointed curves are 
three basic types of morphisms: 
\begin{enumerate}
  \item[(i)]
forgetful morphisms
$$p:\oM_{g,n+1} \rightarrow \oM_{g,n}$$
defined by dropping a marking,
\item[(ii)] irreducible boundary morphisms
$$
q : \oM_{g-1, n+2} \to \oM_{g,n}\ 
$$
defined by identifying two markings to create a node,

\item[(iii)] reducible boundary morphisms
$$
r: \oM_{g_1, n_1+1} \times \oM_{g_2, n_2+1} \to \oM_{g,n}\, , 
$$
where $n=n_1+n_2$ and $g=g_1+g_2$,
defined by identifying the markings of separate pointed curves.
\end{enumerate}
Following \cite[Section 0.1]{faberpandharipande}, the tautological rings{\footnote{Chow groups will be taken with $\mathbb{Q}$-coefficients
unless explicitly stated otherwise.}}
$$R^*(\oM_{g,n}) \subset A^*(\oM_{g,n})$$
are defined as the smallest system of $\mathbb{Q}$-subalgebras (with unit)
closed under push-forward by all morphisms (i)-(iii).
We denote the group of tautological $k$-cycles by
$$R_k({\oM}_{g,n}) = R^{3g-3+n-k}(\oM_{g,n})\, .$$
For an introduction to the current study of tautological classes, we
refer the reader to \cite{faberpantautnottaut,calcmodcurves}.

\subsection{$0$-cycles in the tautological ring}
Whenever the moduli space $\oM_{g,n}$ is rationally connected, we have $$A_0(\oM_{g,n})\stackrel{\sim}{=} \mathbb{Q}\, .$$
Rational connectedness
is known at least in the cases appearing in 
Figure  \ref{Fig:ratconnMgn}. For genus 23 and higher,
$\oM_{g,n}$ is never rationally connected.

\begin{figure}[h]
 \begin{tabular}{|c|c|c|c|c|c|c|c|c|c|c|c|c|c|c|c|c|}
  \hline
  $g$ & 0 & 1 & 2 & 3 & 4 & 5 & 6 & 7&8&9&10&11&12&13&14&15\\
  \hline
  $n_{\mathrm{max}}$ & $\infty$ & 10 & 12 & 14 &15 & 12 & 15 & 11 & 8 & 9 & 3 &10 & 1&0 & 2&0\\
  \hline
 \end{tabular}
 \caption{$\oM_{g,n}$ is rationally connected for $n \leq n_{\mathrm{max}}$, see \cite{benzo,brunoverra, casnatifontanari, farkas, logankodaira, verra}.}
 \label{Fig:ratconnMgn}
\end{figure}

On the other hand, the Chow groups of $0$-cycles 
are of infinite rank as $\mathbb{Q}$-vector spaces
 at least in the following genus 1 and 2 cases (due to
the existence{\footnote{By results of Mumford and Srinivas (see \cite{mumford, roitman, srinivas}
and \cite[Remark 1.1]{grabervakiltaut}), 
the existence of a holomorphic $p$-form for $p \geq 1$ forces
$A_0(\oM_{g,n})$ to have  infinite rank.
Constructions of 
such forms in $g=1$ and $g=2$ are well-known,
see \cite{faberpantautnottaut}.
}}
of holomorphic $p$-forms):
$$A_0({\oM}_{1,n\geq 11})\, , \ \ A_0({\oM}_{2,n\geq 14})\, .$$
Moreover,  such forms{\footnote{There are no written proofs for
the genus 3 and 4 claims, but these expectations, based
on geometric calculations, have been
communicated to us by Faber (in genus 3) and
Farkas (in genus 4).}} 
and infinite ranks are expected
in the following genus $3$ and $4$
cases:
$$A_0({\oM}_{3,n\geq 15})\, ,\ \ A_0({\oM}_{4,n\geq 16})\, .$$
While the data is insufficient for a general prediction, the following 
speculation 
would not be surprising.

\begin{speculation} For $g \geq 1$, the Chow group $A_0({\oM}_{g,n})$ is
of infinite rank except for {\em finitely} many  $(g,n)$.
\end{speculation}

On the other hand, the group 
$R_0(\oM_{g,n})$
of tautological $0$-cycles is much better behaved.
The following result was proven by Graber and Vakil
in \cite{grabervakiltaut} and also 
in \cite{faberpandharipande,Looj}. 

\begin{proposition} For all $(g,n)$, we have
$R_0({\oM}_{g,n}) \stackrel{\sim}{=} \mathbb{Q}$.
\end{proposition}

Since the proof is so short (and depends only upon structural properties
of tautological classes), we present the argument here.\footnote{We follow the path of the proof 
\cite{faberpandharipande,Looj}.
See  \cite[Section 4]{faberpandharipande} and
\cite[Section 5.1]{Looj}.}
Consider 
the moduli space $\oM_{0,2g+n}$ together with the boundary
morphism
$$\iota: \oM_{0,2g+n} \rightarrow \oM_{g,n}$$
defined by pairing the first $2g$ markings to create $g$ nodes.
Since $\oM_{0,2g+n}$ is a rational variety,
$$R_0(\oM_{0,2g+n})\, =\, A_0(\oM_{0,2g+n})\,  
\stackrel{\sim}{=}\,  \mathbb{Q}\, .$$
Therefore, all the moduli points in the image of $\iota$ are
tautological and span a $\mathbb{Q}$-subspace of $R_0(\oM_{g,n})$ of rank 1.
We will prove that the span equals $R_0(\oM_{g,n})$.

Using the additive
generators of the tautological ring constructed in 
\cite[Appendix]{graberpand}, we need
only consider 
$0$-cycles  on $\oM_{g,n}$ which are of a special form. 
The strata of $\oM_{g,n}$ are 
indexed by stable graphs $\Gamma$ of genus $g$ with $n$ markings,
$$\iota_\Gamma : \oM_\Gamma \rightarrow \oM_{g,n}\, .$$
We need only consider $0$-cycles
\begin{equation}\label{h234}
\iota_{\Gamma*} \left[ \prod_{v\in \text{Vert}(\Gamma)} P(v)\right]\,\in R_0(\oM_{g,n})\, ,
\end{equation}
where $P(v)$ is a monomial in $\psi$ and $\kappa$ classes on the 
moduli space $\oM_{g(v),n(v)}$ associated to the vertex $v$. 
Let $\text{deg}(P(v))$ be the degree of the
vertex class.
Using the 
Getzler-Ionel vanishing in the strong form proven{\footnote{See \cite{claderetal} for a much more effective approach to the boundary terms than provided by the argument of  \cite{faberpandharipande}.}} in \cite[Proposition 2]{faberpandharipande},
we can impose the following additional restriction on \eqref{h234}: 
\begin{equation}
\label{d445}
 g(v)>0 \ \ \  \Rightarrow \ \    
\text{deg}(P(v)) <  g - \delta_{0,n(v)}\, .
\end{equation}

Suppose we have a vertex $v$ of $\Gamma$
with $g(v)>0$. Using the vertex stability
condition  $2g(v)-2+n(v)>0$, we deduce
$$
g(v)-\delta_{0,n(v)} \leq 3g(v)-3+n(v)\, .$$
But then we obtain
$$\text{deg}(P(v)) < 3g(v)-3+n(v)\, ,$$
which is impossible since \eqref{h234} is a $0$-cycle.
Therefore, we must have $g(v)=0$ for all $v\in \text{Vert}(\Gamma)$.

The $0$-cycle \eqref{h234} is now easily seen to be in the image of 
\begin{equation}\label{t889}
\iota_*: R_0(\oM_{0,2g+n}) \rightarrow R_0(\oM_{g,n})\, .
\end{equation}
We conclude that the push-forward \eqref{t889} is surjective. \qed

\subsection{Tautological points} \label{Sect:tautpoints}

Our central question here is how to decide whether a given
moduli point
$$[C,p_1,\ldots,p_n]\in \oM_{g,n}$$
determines a tautological $0$-cycle.

While our focus is on the geometry of $C$, there is an interesting connection to arithmetic: Bloch and Beilinson  have conjectured\footnote{See \cite{bbconj3,bbconj1} for the original papers by Bloch and Beilinson and \cite{bbconj2} for a detailed account. See 
\cite[Conjecture 9.12]{bbconj2} 
and the remark thereafter for the particular form of the conjecture
that we have used.} that for a nonsingular proper variety $X$ defined over $\overline{\mathbb{Q}}$, the complex Abel-Jacobi map
\[\Phi_{k,\mathbb{Q}}: A^k_{\mathrm{hom}}(X/\overline{\mathbb{Q}})_\mathbb{Q} \to J^k(X(\mathbb{C}))_{\mathbb{Q}} \]
to the intermediate Jacobian $J^k(X(\mathbb{C}))$ is injective (after tensoring with $\mathbb{Q}$). The map above factors through the usual Abel-Jacobi map of $X(\mathbb{C})$, and the image of  $A^k_{\mathrm{hom}}(X/\overline{\mathbb{Q}})_\mathbb{Q}$ in $A^k_{\mathrm{hom}}(X(\mathbb{C}))_\mathbb{Q}$ is the set of $k$-cycles in $X(\mathbb{C})$ defined over $\overline{\mathbb{Q}}$
which are homologous to $0$. If
the Bloch-Beilinson conjecture holds for
$$X=\oM_{g,n}\, ,$$  
the map
\[\Phi_{3g-3+n, \mathbb{Q}}: A_{0}(\oM_{g,n})_{\mathrm{hom}} \to \mathrm{Alb}(\oM_{g,n}) \otimes \mathbb{Q} \]
would be injective on the set of $0$-cycles defined over $\overline{\mathbb{Q}}$. But since $\oM_{g,n}$ is simply connected
\cite[Proposition 1.1]{BPik}, the Albanese variety is trivial. 
Since a tautological class in $A_{0}(\oM_{g,n})$ can be represented by a curve defined over $\overline{\mathbb{Q}}$, 
we would obtain the following consequence.

\begin{speculation}\label{ddd3}
If the pointed curve $(C,p_1,\ldots,p_n)$ is defined over $\overline{\mathbb{Q}}$,
then the associated moduli point in $A_0(\oM_{g,n})$ is tautological.
\end{speculation}

\noindent A first step in the study of  Speculation \ref{ddd3} is perhaps
 to use Belyi's Theorem to 
express the curve as a Hurwitz covering
$$C\rightarrow \mathbb{P}^1$$ ramified
only over 3 points of $\mathbb{P}^1$. Unfortunately, 
there has not been much progress in the direction of Speculation \ref{ddd3}. However, we will present a result about cyclic covers of $\mathbb{P}^1$ in Section \ref{Sect:cycliccovers}.

\subsection{Curves on surfaces}
Instead of studying the moduli points of special Hurwitz covers of
$\mathbb{P}^1$,
our main results here  concern the moduli points of curves on special {surfaces}.

\vspace{10pt}
\noindent{\bf Rational surfaces}
\vspace{10pt}

Let $S$ be a nonsingular projective rational surface over $\mathbb{C}$, and
let
$C\subset S$ be an irreducible nonsingular curve of genus $g$.
The {\em virtual dimension} in Gromov-Witten theory 
of the moduli space of stable maps $\oM_{g}(S,[C])$ 
is given
by the following formula
$$\mathrm{vdim}\,\oM_{g}(S,[C])= \int_{[C]} c_1(S) + g -1\,.$$
Our first result gives a criterion for curves on rational surfaces
in terms of the virtual dimension.

\begin{theorem} \label{curveratsur}
Let $C\subset S$ be an irreducible nonsingular curve
of genus $g$
on a nonsingular rational surface satisfying
$\int_{[C]} c_1(S)>0$.
 Let $p_1,\ldots, p_n\in C$
be distinct points. If
$$ n\leq \mathrm{vdim}\,\oM_{g}(S,[C])\, ,$$
then $[C,p_1,\ldots,p_n]\in \oM_{g,n}$ determines a tautological
$0$-cycle in $R_0(\oM_{g,n})$.
\end{theorem}

For Theorem \ref{curveratsur}, we always assume $(g,n)$
is in the stable range
$$2g-2+n>0\, .$$
If positivity 
\begin{equation}\label{ff99ff}
\int_{[C]} c_1(S)>0
\end{equation}
holds,
then Theorem \ref{curveratsur} can be applied
with $n=0$ to obtain
$$[C] \in R_0(\oM_g)\, .$$
In case $S$ is toric, positivity \eqref{ff99ff}
always holds for nonsingular curves of genus $g\geq 1$ since
there exists an effective toric anticanonical divisor with affine
complement. Whether positivity  \eqref{ff99ff}
can be avoided in Theorem \ref{curveratsur} is an interesting 
question.{\footnote{The issue is not unrelated to the Harbourne-Hartshorne conjecture
and 
 will be discussed in Section \ref{vvvv}.}}

As an example, consider a nonsingular curve
 $$C_4\subset \mathbb{P}^1\times \mathbb{P}^1\, $$
 of
genus 4 and bidegree $(3,3)$. Positivity \eqref{ff99ff} holds, and
the virtual dimension here is 15,
so all moduli points
$$[C_4,p_1,\ldots,p_{15}] \in \oM_{4,15}$$
are tautological. 
Since the general curve of genus $4$ is of the form $C_4$,
but not all points
of
$\oM_{4,16}$ are expected to be tautological, the virtual dimension bound on $n$ in Theorem \ref{curveratsur}
should not have room for improvement here.

\vspace{10pt}
\noindent{\bf $K3$ surfaces}
\vspace{10pt}

Let $S$ be a nonsingular projective $K3$ surface over $\mathbb{C}$.
Unlike the case of a rational surface, the Chow group $A_0(S,\mathbb{Z})$
of $0$-cycles 
of $S$ is very complicated. However, there is a beautiful 
rank 1 subspace 
$$\mathsf{BV}\subset A_0(S,\mathbb{Z})$$ 
spanned by points lying on rational curves of $S$.
Following \cite{bvclass}, define $p\in S$ to be a {\em Beauville-Voisin} point 
if 
$[p] \in \mathsf{BV}$.

Let
$C\subset S$ be an irreducible nonsingular curve of genus $g$.
The {\em virtual dimension} of the moduli space of stable maps $\oM_{g}(S,[C])$ 
is now
$$\mathrm{vdim}\,\oM_{g}(S,[C])
= g -1\,.$$
Important for us, however, will be the {\em reduced} virtual dimension
$g$.
Our second result gives a criterion for curves on $K3$ surfaces.

\begin{theorem} \label{curvek3}
Let $C\subset S$ be an irreducible nonsingular curve
of genus $g$
on a $K3$ surface. Let $$p_1,\ldots, p_n\in C$$
be distinct 
Beauville-Voisin points of $S$.  If $n\leq g$,
then $[C,p_1,\ldots,p_n]\in \oM_{g,n}$ determines a tautological
$0$-cycle in $R_0(\oM_{g,n})$.
\end{theorem}

For example, consider a nonsingular curve of genus $11$
\begin{equation}\label{k993}
(C_{11}, p_1,\ldots,p_{11})\subset S\, 
\end{equation}
in a primitive class on a $K3$ surface $S$ with 11 distinct points.
By Theorem \ref{curvek3}, 
$$[C_{11}, p_1,\ldots,p_{11}]\in R_0(\oM_{11,11})$$
in case {\em all} the points $p_i$ are Beauville-Voisin.
By the Mukai correspondence \cite{mukaicorr11}, 
we can obtain the general moduli point of $\oM_{11,11}$
by varying the
data \eqref{k993}
in the moduli space of polarized $K3$ surfaces of genus 11 with 11
points.
Since $\oM_{11,11}$ is of Kodaira dimension 19 by \cite[Theorem 5.1]{farkasgenus11},
the Chow group of $0$-cycles is expected (but not known) to
be complicated. In particular, the general moduli point
of $\oM_{11,11}$ is not expected to be tautological. The geometry
of $K3$ surfaces in genus 11 therefore suggests that a
condition on the points
 is necessary.

 The condition of Theorem \ref{curvek3}
exactly links the rank 1 Beauville-Voisin subspace  
$$\mathsf{BV}\subset A_0(S,\mathbb{Z})$$
to the rank 1 tautological subspace 
$$R_0(\oM_{g,n}) \subset A_0(\oM_{g,n})\, .$$

\vspace{5pt}
\noindent{\bf Other surfaces}
\vspace{10pt}

Since every nonsingular curve lies on a
nonsingular algebraic surface, results along the lines
of Theorems \ref{curveratsur} and \ref{curvek3} will always require special surface
geometries.
For nonsingular curves lying on 
Enriques and Abelian surfaces, we hope for results parallel to those in the 
rational and $K3$ surface cases. However,
the questions are, at the moment, open. 
For the Enriques surfaces, there is a clear path, but
the argument depends upon currently open questions about the nonemptiness 
of certain Severi varieties. For Abelian surfaces, the matter
appears more subtle (and there is no obvious line of argument that we
can see). 

For surfaces of general type, canonical curves play a
very special role from the perspective of Gromov-Witten and Seiberg-Witten
theories. A natural question to ask is whether a nonsingular
canonical curve on a surface of general type always determine
a tautological $0$-cycle. We expect new strategies will be required
to resolve such questions in the general type case.

\subsection{Further results on tautological $0$-cycles}
We have seen that a moduli point $$[C,p_1,\ldots, p_n]\in \oM_{g,n}$$
need not determine a tautological $0$-cycle. 
We can measure
{\em how far away from tautological} moduli points of $\oM_{g,n}$
are by considering sums. 
Let $$T(g,n)\in \mathbb{Z}_{>0}$$
 be the smallest number satisfying
the following condition:
{\em for every point $Q_1\in \oM_{g,n}$, there exist $T(g,n)-1$ other points
$Q_2,\ldots,Q_{T(g,n)}\in \oM_{g,n}$ which together have a tautological sum}
$$[Q_1]+[Q_2]+\cdots+[Q_{T(g,n)}] \in R_0(\oM_{g,n})\, .$$

An easy proof of the existence of $T(g,n)$ is given in 
Section \ref{sT}. Finding good bounds for $T(g,n)$ appears much harder.
Our main result here states that the growth 
of $T(g,n)$ for fixed $g$
as $n\rightarrow \infty$ is at most linear in $n$.
Can better asymptotics be found? For example, could $T(g,n)$
for fixed $g$ be bounded independent of $n$?

\subsection{$T$-numbers for $K3$ surfaces}
For comparison, we can consider the parallel question for a $K3$
surface $S$, namely: what is the smallest positive integer $T$ such that for any given $p\in S$  we find $q_2, \ldots, q_T \in S$ such that the sum
\[[p] + [q_2] + \cdots + [q_T] \in A_0(S, \mathbb{Z})\]
lies in the Beauville-Voisin subspace $\mathsf{BV} \subset A_0(S, \mathbb{Z})$?

On the one hand, we have $T \geq 2$, since $T=1$ would be the statement that for every $p \in S$ we have $[p] \in \mathsf{BV}$, a contradiction since $A_0(S, \mathbb{Z})$ is infinite-dimensional and spanned by the classes $[p]$.
On the other hand, since we
have 
families of elliptic curves which sweep out $S$,
the given point $p$ must lie on a (possibly singular)
genus 1 curve $E\subset S$.
Let $R\subset S$ be a rational curve in an ample class.
Since 
$$R\cap E \neq \emptyset\, ,$$
$E$ contains a Beauville-Voisin point $z\in E$. 
We can always solve
the equation
$$[p]+[q] = 2[z]\in A_0(E,\mathbb{Z})$$
for $q\in E$. We conclude that for any $p\in S$, there exists
a $q\in S$ satisfying
$$[p]+[q]\in \mathsf{BV}\, .$$
The $T$-number for $K3$ surfaces is therefore just 2. 

The Hilbert scheme $S^{[n]}$ of $n$ points
on 
$S$ also has a holomorphic form and
a distinguished 
Beauville-Voisin subspace in $A_0(S^{[n]}, \mathbb{Z})$. The holomorphic form shows that the
$T$-number of $S^{[n]}$ is greater than $1$. 
Using families of elliptic curves on $S$, the
$T$-number of $S^{[n]}$ is proven
to be at most $n+1$ in the
upcoming paper \cite{SchYin}, again a linear bound. Whether
the $T$-number is exactly $n+1$ is
an interesting question.

\subsection{Plan of the paper}
We start in Section \ref{bracc}  with basic results about cycles and curves
which we will use throughout the paper. Theorem  \ref{curveratsur} for rational surfaces
is proven in Section \ref{Sect:ratlsurf} and Theorem \ref{curvek3} 
for $K3$ surfaces
is proven in Section \ref{scurvek3}. 
Open questions
for Enriques surfaces, Abelian surfaces, and surfaces of general type 
are discussed in Section \ref{others}. A result concerning cyclic covers of $\mathbb{CP}^1$ is proven in Section \ref{Sect:cycliccovers}. The paper ends with results
about the number $T(g,n)$ in Section \ref{sT}.

\subsection{Acknowledgements}
We thank
C. Faber for contributing to our study of curves and
 G. Farkas for useful conversations about the
birational geometry of moduli spaces. We thank A. Knutsen for discussions about Severi varieties of Enriques surfaces.
Discussions with T. B\"ulles,
A. Kresch, D. Petersen, U. Riess, J. Shen, and Q. Yin
have played an important role. 
We thank the anonymous referee for many helpful comments, improving and
clarifying our exposition.
An early version of the results was
presented at the workshop {\em Hurwitz cycles on the moduli of 
curves} at Humboldt Universit\"at zu Berlin in February 2018.

\section{Basic results about cycles and curves}
\label{bracc}

We start by recalling the following useful (and well-known) result about 
families of algebraic cycles, see  \cite[Proposition 2.4]{voisinunirational}.

\begin{proposition} \label{Prop:spreadout}
Let $\pi : \mathcal{X} \to B$ be a flat morphism of algebraic varieties
 where $B$ is nonsingular of dimension $r$ and let $\mathcal{Z} \in 
A_N(\mathcal{X})$ be a cycle. Then, the set $B_Z$ of points $t \in B$ 
satisfying 
$$\mathcal{Z}_t = \mathcal{Z}|_{\mathcal{X}_t} = 0  \in 
A_{N-r}(\mathcal{X}_t)$$
 is a countable union of proper closed algebraic subsets of $B$.
\end{proposition}

\vspace{10pt}

\begin{proposition}\label{Lem:generic}
 Let $X \subset \oM_{g,n}$ be an irreducible algebraic set such that the generic point of $X$ is tautological. Then, every point of $X$ is tautological.
\end{proposition}
\proof
Consider the trivial family
$$ \pi: \oM_{g,n} \times \oM_{g,n} \rightarrow \oM_{g,n}$$
defined by projection on the second factor.
Let $\Delta \subset \oM_{g,n} \times \oM_{g,n}$ be the diagonal,
and let $S$ be the section of $\pi$ determined by a fixed
tautological point of $\oM_{g,n}$.
By applying{\footnote{We leave the standard
movement  of scheme results to stacks for the reader.}} Proposition \ref{Prop:spreadout} to 
the relative $0$-cycle
$$\mathcal{Z} = \Delta -  S\, ,$$
the set of points in $\overline{M}_{g,n}$ whose class is tautological is a countable union of closed algebraic sets. Since the generic point of $X$ is contained in this union, $X$ must also be  contained.
\qed

\vspace{10pt}

Let $S$ be a nonsingular projective surface which is either rational or $K3$. 
In both cases,
$$  \mathrm{Pic}(S) = H_2(S,\mathbb{Z})\, .$$
Let $L \in \mathrm{Pic}(S)$ be an effective divisor class. 
Let $|L| = \mathbb{P}(H^0(S,L))$ be the associated
linear system of divisors with hyperplane class
$H \in A^1(|L|)$. There exists a natural Hilbert-Chow morphism
\begin{align} \label{eqn:classmorphism}
 c: \oM_{g,n}(S,c_1(L)) &\to |L|\, ,
\end{align}
sending a stable map $ (f:(C,p_1, \ldots, p_n) \to S)$ to the effective divisor $ f_* [C]$. 

In the stable range $2g-2+n>0$,
let $$\epsilon: \oM_{g,n}(S,c_1(L)) \to \oM_{g,n}$$ be the natural forgetful morphism. Let 
$$\mathrm{ev}_i : \oM_{g,n}(S,c_1(L)) \to S$$ be the evaluation map corresponding to the $i$th marking.

\begin{lemma}\label{Lem:clmorph}
 Let $S$ be a 
rational  surface with $L \in \mathrm{Pic}(S)$.
Let $C \subset S$ be a nonsingular irreducible
 curve of genus $g$ contained in $|L|$. Assume 
 \[\dim |L| = \mathrm{vdim}\,\oM_{g}(S,[C]) = g-1 + 
\int_{[C]} c_1(S) \, .\]
Then, for   $0 \leq n \leq \mathrm{vdim}\,\oM_{g}(S,[C])$ satisfying $2g-2+n>0$ 
 and pairwise distinct points
 $p_1, \ldots, p_n \in C$, we have 
 \begin{equation} \label{eqn:Chow1}
  \epsilon_* \left( 
c^* H^{\dim |L|-n} \cap \prod_{i=1}^n \mathrm{ev}_i^* [p_i] 
\cap [\oM_{g,n}(S,[C])]^{\mathrm{vir}} 
\right) = [C,p_1, \ldots, p_n] 
 \end{equation}
in $A_0(\oM_{g,n})$.
\end{lemma}
\proof
 We first  prove the Lemma for general points 
$$p_1, \ldots, p_n \in C\, .$$
 For general points $p_i$, the set of curves in $|L|$ 
passing through the $p_i$ is a linear subspace $H_1$ of codimension $n$. 
We choose a complementary linear subspace $H_2 \subset |L|$ of codimension $r-n$ 
satisfying 
$$H_1 \cap H_2 = \{[C]\}\, .$$
Therefore,  on $\oM_{g,n}(S,[C])$, the cycle $c^* [H_2] \cap \prod_{i=1}^n \mathrm{ev}_i^* [p_i]$ is supported on the point 
\begin{equation}\label{wwppww}
[(C,p_1, \ldots, p_n) \hookrightarrow S] \in \oM_{g,n}(S,[C])
\end{equation}

Near the point \eqref{wwppww} in
 $\oM_{g,n}(S,[C])$, the map $\Phi=(c,\mathrm{ev}_1, \ldots, \mathrm{ev}_n)$ defines a local isomorphism{\footnote{Since all the curves
$D$ near $C$ are irreducible and nonsingular, the inverse map is
well-defined.}} 
to the incidence variety
 \[\mathcal{I} = \big\{(D,q_1, \ldots, q_n) : D \in |L|\, , q_1, \ldots, q_n \in D\big\} \subset |L| \times S^n.\]
 Since near \eqref{wwppww} $\mathcal{I}$ is nonsingular
 of dimension $\dim\,|L|+n$ and since this is
the virtual dimension of $\oM_{g,n}(S,[C])$,  
the virtual fundamental class
restricts to the standard fundamental class near \eqref{wwppww}.
Since $H_2 \times \prod_{i=1}^n [p_i]$ intersects $\mathcal{I}$ transversally in the point $([C],p_1, \ldots, p_n)$,  we obtain the equality
\eqref{eqn:Chow1}. 
 
 We finish the proof by going from the case of general points $p_1, \ldots, p_n \in C$ to the case of any pairwise distinct set of points. 
Consider the complement $B=C^n \setminus \Delta$ of the diagonals inside the product $C^n$. The difference of the two sides of equation (\ref{eqn:Chow1}) defines a natural cycle $\mathcal{Z}$ inside $\oM_{g,n} \times B$. 
For $b \in B$ general, we have $$\mathcal{Z}|_{\oM_{g,n} \times \{b\}}=0\, .$$ 
By Proposition \ref{Prop:spreadout}, the set of such $b$ is a countable union of closed algebraic sets, and so must be all of $B$.
\qed

\vspace{10pt}

For $S$ a nonsingular projective 
$K3$ surface, we need a variant of Lemma 
\ref{Lem:clmorph} involving the reduced virtual fundamental class (see \cite{bryanleung, maulikpandharipande}). 
\begin{lemma}\label{Lem:clmorphk3}
 Let $S$ be a $K3$ surface with $L\in \mathrm{Pic}(S)$.
Let $C \subset S$ be a nonsingular irreducible
 curve of genus $g$ contained in $|L|$. 
 Then for 
$$0 \leq n \leq g \ \ \text{satisfiying} \ \ 2g-2+n>0$$
 and distinct points $p_1, \ldots, p_n \in C$, we have 
 \begin{equation} \label{eqn:Chow1k3}
  \epsilon_* \left(  c^* H^{g-n} \cdot \prod_{i=1}^n \mathrm{ev}_i^* [p_i] 
\cap [\oM_{g,n}(S,[C])]^{\mathrm{red}}
\right) = [C,p_1, \ldots, p_n]
 \end{equation}
 in $A_0(\oM_{g,n})$. 
\end{lemma}
\proof
Since $L=\mathcal{O}_S(C)$, 
the exact sequence
$$ 0 \rightarrow H^0(S, \mathcal{O}_S) \rightarrow
H^0(S, \mathcal{O}_S(C)) \rightarrow H^0(C,\mathcal{O}_C(C)) \rightarrow
H^1(S,\mathcal{O}_S)\rightarrow \,\ldots $$
together with the ranks
$$h^0(C,\mathcal{O}_C(C))= h^0(C, \omega_C)=g\, , \ \ \ h^1(S,\mathcal{O}_S)=0$$
shows $\dim|L|=g$.
Hence, we have
 \[\dim\,|L|+n = g+n = \mathrm{dim}\,[\oM_{g,n}(S,c_1(L))]^{\mathrm{red}}\, .\] 
The proof of Lemma \ref{Lem:clmorph} can then be exactly followed
for the reduced class here to conclude the result.
\qed


\section{Rational surfaces}
\label{Sect:ratlsurf}

\subsection{Proof of Theorem \ref{curveratsur}}

If $C$ is of genus $g=0$, Theorem \ref{curveratsur}  
is trivial (since the moduli space $\oM_{0,n}$ is rational and
all $0$-cycles are tautological). We will assume $g\geq 1$.
The argument proceeds in three steps:
\begin{enumerate}
\item[(1)]
We apply Lemma \ref{Lem:clmorph} to express the 0-cycle
 $$[C,p_1, \ldots, p_n]\in A_0(\oM_{g,n})$$
 in terms of a push-forward involving the virtual fundamental class of 
$\oM_{g,n}(S,[C])$. 
\item[(2)]
We deform the rational surface $S$ to a nonsingular projective 
toric surface $\widehat{S}$ over a base which is rationally connected. 
\item[(3)] We apply virtual localization \cite{virtloc} to 
the toric surface $\widehat{S}$
  to conclude the  desired class is tautological.
\end{enumerate}
 
\noindent  \textbf{Step 1.} To apply Lemma \ref{Lem:clmorph}, we must
check the hypothesis
\begin{equation}\label{kk332}
 \dim |L|= g-1 + \int_{[C]} c_1(S)\, , 
\end{equation}
where $L=\mathcal{O}_S(C)$.
Condition \eqref{kk332}
 is equivalent to $h^0(L) = g+\int_{[C]} c_1(S)$.

 Since $C$ is nonsingular of genus $g$, the adjunction formula yields
\[ \langle[C], [C] - c_1(S)\rangle  = 2g-2\, ,\]
where $\langle,\rangle$ is the intersection product on $S$.
 On the other hand, by Riemann-Roch we have
 \begin{align*}
  \chi(L) &= \frac{1}{2} \langle [C], [C] + c_1(S) \rangle + \chi(\mathcal{O}_S)\\
  &= \frac{1}{2} \langle [C], [C] - c_1(S) \rangle + \langle [C], c_1(S) \rangle + 1\\
  &= g-1 + \langle [C], c_1(S) \rangle + 1\\& = g + \langle [C], c_1(S) 
\rangle\, .
 \end{align*}
Furthermore, we have
$$h^2(L)= h^2(S,\mathcal{O}_S(C))= h^0(S,\omega_S(-C))
\leq h^0(S,\omega_S)=0\, ,$$
where the last equality holds since $S$ is rational. 
So, we see 
$$h^0(L) \geq \chi(L)=g+\langle [C], c_1(S) \rangle\, .$$

To prove the vanishing of $h^1(L)$, we use the sequence
$$H^1(S, \mathcal{O}_S) \rightarrow H^1(S,\mathcal{O}_S(C)) \rightarrow
H^1(C,\mathcal{O}_C(C)) \rightarrow H^2(S,\mathcal{O}_S)\, .$$
Since the higher cohomologies of $\mathcal{O}_S$ on $S$
vanish,
$$ h^1(S,L)= h^1(C, \mathcal{O}_C(C))\, .$$
By Serre duality and adjunction, we have
$$h^1(C, \mathcal{O}_C(C))= h^0(C,\omega_C(-C)) =
h^0(C,\omega_S)\, .$$
However, by the positivity hypothesis, 
$$\langle [C], c_1(\omega_S) \rangle < 0\, ,$$
so $h^0(C,\omega_S)=0$. 
 
 Since the hypotheses  of Lemma \ref{Lem:clmorph} hold,
 we may apply the conclusion:
for $r=g-1 + \int_{[C]} c_1(S)$ and pairwise distinct $p_1, \ldots, p_r \in C$, we have 
 \begin{equation} \label{eqn:Chow2}
  \epsilon_* \left(  
\prod_{i=1}^r \mathrm{ev}_i^* [\mathrm{pt}] \cap
[\oM_{g,r}(S,c_1(L))]^{\mathrm{vir}}
 \right) = [C,p_1, \ldots, p_r] \in A_0(\oM_{g,r})\, ,
 \end{equation}
 where $[\mathrm{pt}] \in A_0(S,\mathbb{Z})$ is the class of (any) point
as  $S$ is rational.

\vspace{10pt}
\noindent \textbf{Step 2.} The rational surface $S$ can be deformed to a toric surface 
$\widehat{S}$ in a smooth family $$\mathcal{S} \to B$$ 
over a rationally connected variety $B$ containing $S, \widehat{S}$ as special fibres.{\footnote{There
is no difficultly in finding such a deformation. The minimal model of $S$
is toric. The exceptional divisors can then be moved to toric fixed points.}}
The line bundle $L$ can be deformed along with $S$ to a line bundle
$$\widehat{L} \rightarrow \widehat{S}\, .$$
Since the virtual fundmental class is constructed in families \cite{behrendfantechi},
 $$\epsilon_* \left(  
\prod_{i=1}^r \mathrm{ev}_i^* [\mathrm{pt}] \cap
[\oM_{g,r}(S,c_1(L))]^{\mathrm{vir}}
 \right) =
 \epsilon_* \left(  
\prod_{i=1}^r \mathrm{ev}_i^* [\mathrm{pt}] \cap
[\oM_{g,r}(\widehat{S},c_1(\widehat{L}))]^{\mathrm{vir}}
 \right)\, .$$
We have therefore moved the calculation to the toric setting.

\vspace{10pt}
\noindent
 \textbf{Step 3.} The
virtual localization formula of \cite{virtloc}
applied to the toric surface $\widehat{S}$ immediately shows
\[\epsilon_* \left( 
\prod_{i=1}^n \mathrm{ev}_i^* [\mathrm{pt}]
\cap 
[\oM_{g,r}(\widehat{S},c_1(\widehat{L}))]^{\mathrm{vir}}
   \right) \in R_0(\oM_{g,r})\, .\]

We have proven that the $0$-cycle $[C,p_1, \ldots, p_r] \in A_0(\oM_{g,r})$
is tautological. If $0\leq n\leq r$, 
$$[C,p_1, \ldots, p_n] \in A_0(\oM_{g,n})$$
must also be tautological (by applying the forgetful map). \qed

\subsection{Variations}\label{vvvv}
Let $S$ be a nonsingular projective rational surface, and let
$$C\subset S$$
be a reduced, 
irreducible, nodal curve of arithmetic genus $g$
satisfying the positivity condition 
\begin{equation}\label{fff56}
\int_{[C]} c_1(S) >0\,.
\end{equation}
The statements and proofs of Lemma \ref{Lem:clmorph} 
and Theorem \ref{curveratsur}
are still valid for such curves\footnote{The points $p_i$ here
are distinct and lie in the nonsingular locus of $C$.} : 
{\em the 0-cycle 
$$[C,p_1,\ldots, p_n] \in A_0(\oM_{g,n})$$
is tautological if $n\leq\mathrm{vdim}\,\oM_{g}(S,[C])$}.

Can the positivity condition \eqref{fff56} be relaxed?
Positivity was used in the proof of Theorem \ref{curveratsur}
only to prove that the associated linear series has the
expected dimension. 
If $C\subset S$ is an irreducible nodal curve of arithmetic genus $g$
satisfying
\begin{equation}\label{nnee3}
h^1(S,\mathcal{O}_S(C))=0\, ,
\end{equation}
then we can {\em still} conclude that 
the 0-cycle 
$$[C,p_1,\ldots, p_n] \in A_0(\oM_{g,n})$$
is tautological if $n\leq \mathrm{vdim}\,\oM_{g}(S,[C])$.

According to the Harbourne-Hirschowitz conjecture \cite{HH1,HH2},
the vanishing \eqref{nnee3} should always hold if
$S$ is sufficiently general. We therefore 
expect an affirmative answer to the following question.

\begin{question} Let 
$C\subset S$ be an irreducible nonsingular (or an irreducible nodal) 
curve  with {\em no}
positivity assumption on $\int_{[C]} c_1(S)$. Is 
the 0-cycle 
$$[C,p_1,\ldots, p_n] \in A_0(\oM_{g,n})$$
 tautological for $n\leq \mathrm{vdim}\,\oM_{g}(S,[C])$?
\end{question}

\vspace{5pt}
On the other hand, if
 $C\subset S$ is  a {\em reducible} nodal curve, 
we obtain a parallel statement by applying the results above for each irreducible component separately. Here, each component $C_v$ with arithmetic genus $g_v$ must satisfy the positivity condition \eqref{fff56}, and the number of markings plus the number of preimages of nodes must be bounded by the virtual dimension $\mathrm{vdim}\,\oM_{g_v}(S,[C_v])$.

\section{K3 surfaces}
\label{Sect:K3} \label{scurvek3}
\subsection{Beauville-Voisin classes}
On a nonsingular projective $K3$ surface $S$,
 there exists a canonical zero cycle $c_S \in A_0(S,\mathbb{Z})$ of degree $1$ 
satisfying the following three properties \cite{bvclass}:
\begin{itemize}
    \item all points in $S$ lying on a (possibly singular) rational curve have class $c_S \in A_0(S,\mathbb{Z})$,
    \item the image of the intersection product $A_1(S,\mathbb{Z}) 
\otimes A_1(S,\mathbb{Z}) \to A_0(S,\mathbb{Z})$ lies in $\mathbb{Z} \cdot c_S$,
    \item the second Chern class $c_2(S)$ is equal to $24 c_S$.
\end{itemize}
The {\em Beauville-Voisin} subspace is defined by
$$\mathsf{BV}= \mathbb{Z}\cdot c_s \subset A_0(S,\mathbb{Z})\, .$$
A point $p\in S$ is a {\em Beauville-Voisin} point if
$[p]\in \mathsf{BV}$.

\subsection{Proof of Theorem \ref{curvek3}}
The claim is trivial for genus $g=1$ 
since $\oM_{1,1}$ is rational.
We can therefore assume $g \geq 2$. By Lemma \ref{Lem:clmorphk3}, we have 
 \begin{equation} \label{eqn:Chow2k3}
  \epsilon_* \left(  c^* H^{g-n} \cap \prod_{i=1}^n \mathrm{ev}_i^* [p_i] 
\cap [\oM_{g,n}(S,[C])]^{\mathrm{red}}
\right) = [C,p_1, \ldots, p_n]
 \end{equation}
in $A_0(\oM_{g,n})$.

We briefly recall the notation used in \eqref{eqn:Chow2k3}.
For $L=\mathcal{O}_S(C)$,
 $$c: \oM_{g,n}(S,[C]) \to \mathbb{P}(H^0(S,L))$$
 is the map sending 
$$f:(\widehat{C},\widehat{p}_1, \ldots, \widehat{p}_n) \to S$$
to $f_*[\widehat{C}] \in |L|$ and 
$H$ is the hyperplane class of $|L|$.
Since the points $p_i$ are all Beauville-Voisin,
 equality (\ref{Lem:clmorphk3}) immediately implies that the right hand side
 depends {\em only}
  upon the surface $S$ and the class 
$$\left[\mathcal{O}_S(C)\right] \in \mathrm{Pic}(S)\, .$$

 By Lemma 2.3 of \cite[Chapter 2]{lectk3}, the line bundle $L$ 
is base point free and hence nef. 
Let $$L = L_0^{\otimes k}\, , \ \ \ k\geq 1$$ for 
$L_0$ primitive of degree 
$$d=2 g' - 2>0\, .$$
Then, $L_0$ is still nef, so $(S,L_0)$ is a quasi-polarized $K3$
 surface of degree $d$. Consider the moduli stack 
$\mathcal{F}_{d}$ of quasi-polarized $K3$ surfaces $(\widehat{S},\widehat{L}_0)$ of degree $d$. 
Let $$\pi: \mathcal{S} \to \mathcal{F}_d$$
 be the universal $K3$ surface over $\mathcal{F}_d$ with universal polarization $\mathcal{L}_0\in A^1(\mathcal{S})$. 
The restriction of $(\mathcal{S},\mathcal{L}_0)$ to the fibre 
over $(\widehat{S},\widehat{L}_0) \in \mathcal{F}_d$ is isomorphic to $(\widehat{S},\widehat{L}_0)$, 
see \cite{K3tautological}. 
 
 Consider
 furthermore the projective bundle
$$\mathcal{P} = \mathbb{P}(R^0 \pi_*((\mathcal{L}_0)^{\otimes k}))
\rightarrow  \mathcal{F}_d$$
 parametrizing elements in the linear system $(\mathcal{L}_0)^{\otimes k}$ on the fibres of $\mathcal{S}$. The projective bundle $\mathcal{P}$
is of fibre dimension $g$ by Theorem 1.8 of
\cite[Chapter 2]{lectk3}.
 
 We can then 
obtain the left hand side of \eqref{Lem:clmorphk3}
 as a fibre in a family of cycles parametrized by $\mathcal{F}_d$. 
Indeed, denote by $\mathcal{S}^n$ the $n$-fold self product of $\mathcal{S}$ 
over  $\mathcal{F}_d$ and consider the following commutative diagram:
 
 \begin{equation}
  \begin{tikzcd}
   & \oM_{g,n}(\pi,c_1(\mathcal{L}_0^k)) \arrow[r,"c\times \mathrm{ev}"]\arrow[dl,swap,"\epsilon"] \arrow[d,"\pi"] & \mathcal{P} \times_{\mathcal{F}_d} \mathcal{S}^n \arrow[dl]\\
   \oM_{g,n} & \mathcal{F}_d & 
  \end{tikzcd}.
 \end{equation}
 Here, $\oM_{g,n}(\pi,c_1(\mathcal{L}_0^k))$ 
is the moduli space of stable maps to the fibres of $\pi$ of curve class equal to $c_1(\mathcal{L}_0^k)$ on the fibres of $\pi$. 
The map $c$ is the version of the previous map $c$ in families, and 
$$\mathrm{ev}=(\mathrm{ev}_1, \ldots, \mathrm{ev}_n)$$
 is the evaluation map corresponding to the $n$ points. 
Let $$\mathcal{H}=c_1(\mathcal{O}_\mathcal{P}(1))$$ be the hyperplane class of the projective bundle $\mathcal{P}$, and 
let 
$$c_\mathcal{S} = \frac{1}{24} c_2( \Omega_\pi) \in A^2(\mathcal{S})$$
 be the relative Beauville-Voisin class of the family $$\pi : \mathcal{S} \to \mathcal{F}_d\, .$$
Consider the cycle $\mathcal{Z} \in  A^{3g-3+n}(\oM_{g,n} \times \mathcal{F}_d)$ defined by
 \begin{equation}
     \mathcal{Z} = (\epsilon,\pi)_* \left(
 c^* \mathcal{H}^{g-n} \cap \prod_{i=1}^n \mathrm{ev}_i^* c_\mathcal{S} 
\cap [\oM_{g,n}(\pi,c_1(\mathcal{L}_0^k)]^{\mathrm{red}}
\right)\, .
 \end{equation}
The fibre of $\mathcal{Z}$ over $(S,L_0)$ is equal to the left hand side of 
\eqref{Lem:clmorphk3}.

By Proposition \ref{Prop:spreadout}, we need only
show that the fibre of $\mathcal{Z}$ over the general point of $\mathcal{F}_d$ is tautological. 
So let 
$$(\widehat{S},\widehat{L}_0) \in \mathcal{F}_d$$
 be a general quasi-polarized $K3$ of degree $d$. By 
the existence result of \cite{chenk3}, 
the linear system $\big|\widehat{L}_0^{\otimes k}\big|$ contains an 
irreducible nodal rational curve 
$$R\subset \widehat{S}\, .$$ 
Furthermore, since $(\widehat{S},\widehat{L})$ 
is general, we can assume that $\widehat{L}_0$ 
and thus $\widehat{L}_0^{\otimes k}$ are basepoint free
 (see Theorem 4.2 of \cite[Chapter 2]{lectk3}).
By Bertini's theorem, the general member $\widehat{C}$ of
 the linear system $\big|\widehat{L}_0^{\otimes k}\big|$ 
intersects the rational curve $R$ only in reduced points. 
The number of these intersection points is exactly 
$$\langle [\widehat{C}], [\widehat{C}] \rangle = 2g-2\, ,$$ 
which is at least $g$
 (since we assume $g \geq 2$).
Choose distinct points 
$$q_1, \ldots, q_n \in R \cap \widehat{C}\, .$$
 Certainly all the $q_i$ are Beauville-Voisin points
since they lie on the rational curve $R$. 
Since $$R,\widehat{C} \in |\widehat{L}_0^k|\, ,$$ 
there exists a pencil of curves connecting $(\widehat{C},q_1, \ldots, q_n)$ and $(R,q_1, \ldots, q_n)$. 
The 0-cycle given by $[(R,q_1, \ldots, q_n)]\in A^0(\oM_{g,n})$ is clearly tautological, 
since the point lies in the image of
 $$\oM_{0,n+2g}\rightarrow \oM_{g,n}\, .$$
 Therefore, $[(\widehat{C},q_1, \ldots, q_n)]$ is tautological. \qed

\vspace{10pt}

We isolate part of the above proof as a separate corollary for
later application.

\begin{corollary} \label{Cor:k3augmented}
 Let $S$ be a $K3$ surface with $L\in \mathrm{Pic}(S)$.
There exists a $\mathbb{Q}$-linear map 
$$\Phi: A_0(S^n) \to A_0(\oM_{g,n})$$ 
defined by 
$$\Phi(\alpha) =\epsilon_* \left(  c^* H^{g-n} \cap  \mathrm{ev}^* \alpha\cap
[\oM_{g,n}(S,c_1(L))]^{\mathrm{red}}  \right)\, .$$
For an irreducible
 nonsingular projective curve $C \subset S$ of genus $g \geq 1$ in the
linear series $L$
and distinct points $p_1, \ldots, p_n \in C$ we have
 \begin{equation}
     [(C,p_1, \ldots, p_n)] = \Phi( [(p_1, \ldots, p_n)] ).
 \end{equation}
 Moreover,  $\Phi((c_S)^{\times n})$ is tautological.
\end{corollary}

\subsection{Quotients}

The symmetric group $S_n$ acts on $\oM_{g,n}$ by permuting the markings.
 For a partition $\mu=(n_1, \ldots, n_\ell)$ of $n$, 
let $$S_\mu = S_{n_1} \times \cdots \times S_{n_\ell} \subset S_n$$
be the subgroup
 permuting elements within the blocks defined by $\mu$. 
The stack quotient $$\oM_{g,\mu} = \oM_{g,n}/S_\mu$$ 
parametrizes curves 
\[\big(C,(\{p_{i,1}, \ldots, p_{i,n_i}\})_{i=1, \ldots, \ell}\big)\]
together with $\ell$ pairwise disjoint sets of marked points with sizes $n_i$ according to the partition $\mu$. 
The quotient map 
$$\pi : \oM_{g,n} \to \oM_{g,\mu}$$ 
allows us to define the tautological ring $R^*(\oM_{g,\mu})$ as the 
image of $R^*(\oM_{g,n})$ via push-forward by $\pi$.
The composition $$\pi_* \pi^* : A^*(\oM_{g,\mu}) \to A^*(\oM_{g,\mu})$$ is given by multiplication by $|S_\mu|$ . 
Therefore, to check if a cycle $\alpha$ on $\oM_{g,\mu}$ is tautological, 
it suffices to check that $\pi^*(\alpha)$ is tautological on $\oM_{g,n}$.

The following result for the quotient
moduli spaces $\oM_{g,\mu}$ is parallel to   Theorem \ref{curvek3} for
 $\oM_{g,n}$.
\begin{theorem} \label{Thm:K3markquot}
Let $C\subset S$ be an irreducible nonsingular curve of
genus $g$ on a $K3$ surface.
Let $0 \leq n \leq g$ and fix a partition $\mu=(n_1, \ldots, n_\ell)$ of $n$. 
Let 
$$(p_{i,j})_{\substack{i=1, \ldots, \ell \\ j=1, \ldots, n_i}}$$ 
be a collection
of distinct points $p_{i,j} \in C$ satisfying
 \begin{equation}
     [p_{i,1}] + [p_{i,2}] + \cdots + [p_{i,n_{i}}] \in \mathsf{BV}
 \end{equation}
for all $1\leq i \leq \ell$.
Then, the $0$-cycle
 \[[C,(\{p_{i,1}, \ldots, p_{i,n_i}\})_{i=1, \ldots, \ell}\, ] \in A_0(\oM_{g,\mu})\]
 is tautological.
\end{theorem}

\proof
It suffices to show that
the pullback $\pi^*([C,({p_{i,j}})_i])$ is tautological. 
Fix  an ordering $\ul p = (p_{i,j})_{i}$ of all the
markings. The pullback is exactly given by
\[\pi^*([C,({p_{i,j}})_i\, ]) = \sum_{\sigma \in S_\mu} [\sigma(C,\ul p)]\, .\]
Using Corollary \ref{Cor:k3augmented}, we can write the result
 as $\Phi(\Sigma(\ul p))$ for the sum
\[\Sigma(\ul p) = \sum_{\sigma \in S_\mu} [\sigma.\ul p] \in A_0(S^n)\, ,\]
where we have used the natural permutation action of $S_n$ on $S^n$. 

We claim that the cycle $\Sigma(\ul p)$ only depends on the blockwise sums 
$$\Sigma_i(\ul p)=\sum_j [p_{i,j}] \in A_0(S)$$
 for $i=1, \ldots, \ell$. 
Blockwise dependence together
 with the hypothesis $$\Sigma_i(\ul p)\in \mathsf{BV}$$ immediately 
yields the result of Theorem \ref{Thm:K3markquot}
(since we can exchange all the $p_{i,j}$ for Beauville-Voisin points).

It remains only to prove the blockwise dependence.
We first observe that we can write 
$\Sigma(\ul p)$
 as a product
\[\Sigma(\ul p) = \sum_{\sigma \in S_\mu} [\sigma.\ul p] = \prod_{i=1}^\ell \sum_{\sigma_i \in S_{n_i}} [\sigma_i. (p_{i,j})_{j=1, \ldots, n_i}]\, ,\]
where we recall that $S_\mu$ is the product of the groups $S_{n_i}$. 
It suffices then to show that the $i$th factor in the above product only depends on the sum $\Sigma_i(\ul p)$. The latter claim amounts to
a reduction to the case of the partition $\mu=(n)$ 
where all the markings are permuted. 

Let $P=\{p_1, \ldots, p_n\}$. We will  write 
$$\Sigma(\ul p) = \sum_{\sigma \in S_n} [(p_{\sigma(1)}, \ldots, p_{\sigma(n)})]$$
 as a sum of terms depending only upon 
$$\theta = \Sigma_1(\ul p) = [p_1] + \cdots + [p_n]\, $$
using a simple inclusion-exclusion strategy.

We illustrate the strategy in the case of  $n=3$. 
We start with the formula
\[\theta^{\times 3} = \sum_{q_1, q_2, q_3 \in P} [(q_1,q_2,q_3)]\, .\]
To obtain $\Sigma(\ul p)$, we must 
substract all summands where there is a pair $i \neq j$ with $q_i = q_j$. 
Let $$\Delta_{12,3}\, ,\  \Delta_{13,2}\, ,\  \Delta_{23,1}\, : S^2 \to S^3$$ 
be the three diagonal maps. The cycle
\[\theta^{\times 3} - (\Delta_{12,3})_*(\theta^{\times 2}) - (\Delta_{13,2})_*(\theta^{\times 2})- (\Delta_{23,1})_*(\theta^{\times 2})\]
is equal to $\Sigma(\ul p)$ minus $2$ times the cycle 
$$[(p_1,p_1,p_1)]+[(p_2,p_2,p_2)]+[(p_3,p_3,p_3)]\, .$$ 
 We can cancel the error term by adding a correction
$2(\Delta_{123})_*(\theta)$ by the small diagonal:
\[\Sigma(\ul p) = \theta^{\times 3} - (\Delta_{12,3})_*(\theta^{\times 2}) - (\Delta_{13,2})_*(\theta^{\times 2})- (\Delta_{23,1})_*(\theta^{\times 2}) + 2(\Delta_{123})_*(\theta)\, .\]
Such an inclusion-exclusion strategy is valid for 
all $n\geq 1$.
\qed


\section{Other surface geometries} \label{others}
\subsection{Enriques surfaces}
An Enriques surface $E$ is a free $\mathbb{Z}_2$
quotient of a nonsingular projective $K3$ surface $S$:
$$E = S / \mathbb{Z}_2\, .$$

\begin{conjecture} \label{ggef} The moduli point of
an irreducible nonsingular curve $C\subset E$
of genus $g\geq 2$ determines a tautological $0$-cycle in
$\oM_g$.
\end{conjecture}

There is a clear strategy for the proof of Conjecture \ref{ggef}.
The curve $C$ is expected to move in a linear series $|L|$ on $E$ of
dimension $g-1$. We therefore expect to find irreducible curves
$\widehat{C}\in |L|$ with $g-1$ nodes. 
The issue can be formulated as the nonemptiness of 
 certain Severi varieties for linear
systems on Enriques surfaces which is
currently being studied, see \cite{severienriques}.
Once it is shown that the linear series $|L|$ contains an irreducible
$(g-1)$-nodal curve $\widehat{C} \subset E$, the final step is to prove that
the $0$-cycle 
$$[\widehat{C}]\in A_0(\oM_g)$$
is always tautological. In fact, the following stronger result
holds.

\begin{proposition}\label{eeee}
The locus of irreducible $(g-1)$-nodal curves in $\oM_{g,1}$ 
 is rational.
 In particular, every such curve defines a tautological cycle 
$$[\widehat{C},p] \in R_0(\oM_{g,1})\, .$$
\end{proposition}
\proof
The closure of the locus of $(g-1)$-nodal curves is parametrized by the gluing map
\[\xi : \oM_{1,1+2(g-1)} \to \oM_{g}\]
taking a curve $(X,p,q_1, q_1', \ldots, q_{g-1}, q_{g-1}')$ of genus $1$ with $1+2(g-1)$ markings and identifying the $g-1$ pairs $q_j, q_j'$ of points. 
The group 
$$G=(\mathbb{Z}/2\mathbb{Z})^{g-1} \rtimes S_{g-1}$$ 
acts on $\oM_{1,1+2(g-1)}$: 
the $j$th factor $\mathbb{Z}/2\mathbb{Z}$ switches the two points $q_j, q_j'$ and the group $S_{g-1}$ permutes the $n$ pairs of points among each other. 
Since the gluing map $\xi$ is invariant under this action, it factors through the map
\[\tilde \xi : \oM_{1,1+2(g-1)}/G \to \oM_{g,1},\]
which is birational onto its image.

To prove $\mathcal{M}=\cM_{1,1+2(g-1)}/G$ is rational, 
we take a  modular reinterpretation. 
Instead of remembering the $2(g-1)$ points $q_j, q_j'$ on $X$ individually, 
we only remember the set 
$$\{ D_j=q_j + q_j' : j=1, \ldots, g-1 \}$$
 of $g-1$ effective divisors of degree $2$ on the curve $X$. 
We therefore have a birational identification
\begin{equation*}
    \mathcal{M} \mathrel{\leftrightarrow} \left\{\left(X,p, (D_j)_{j=1}^{g-1}\right): \begin{array}{c} X \text{ nonsingular elliptic curve with origin $p$,} \\  
    D_j \subset X \text{ effective degree $2$ divisors}\end{array} \right\} / S_{g-1}\, ,
\end{equation*}
where $S_{g-1}$ acts by permuting the divisors $D_1, \ldots, D_{g-1}$. 

An effective divisor $D_j \subset X$ is equivalent to the data of the degree $2$ line bundle 
$$\mathcal{L}_j=\mathcal{O}(D_j)$$ 
together with an element 
$$s_j \in \mathbb{P}(H^0(X, \mathcal{L}_j)) \cong \mathbb{P}^1\, .$$
 Furthermore, the class of the line bundle $\mathcal{L}_j$ is equivalent to specifying a point $l_j \in X$, by the correspondence sending 
$l_j$ to $\mathcal{O}(p+l_j)$, where $p \in X$ is the origin.
We define
\begin{equation*}
     \mathcal{P}=\left\{\left(X,p, (l_j)_{j=1}^{g-1}, (s_j)_{j=1}^{g-1}\right): \begin{array}{c} X\text{ nonsingular elliptic curve with origin $p$},\\ 
    l_j \in E\\  
    s_j \in \mathbb{P}(H^0(X, \mathcal{O}(p+l_j))) \end{array} \right\}\ .
\end{equation*}
We have a birational identification 
$$\mathcal{M} \mathrel{\longleftrightarrow} \mathcal{P}/ S_{g-1}\, .$$

 By forgetting the projective sections $s_j$, we obtain a map 
$$\mathcal{P} \to \mathcal{S}$$ to the space $\mathcal{S}$ parametrizing tuples 
$(X,p, (l_j)_j)$ as above. 
The above forgetful map 
is a $(\mathbb{P}^1)^{g-1}$-bundle which descends (birationally)
to a  $(\mathbb{P}^1)^{g-1}$-bundle 
$$\mathcal{P}/S_{g-1} \to \mathcal{S}/S_{g-1}$$ on the quotient. The base, the moduli space
parameterizing the data
$$(X,p, (l_j)_j)$$ up to
permutations of the $l_j$ by $S_{g-1}$, 
is easily seen to be rational using, to start, the 
rationality of the universal family of $\text{Jac}_2$ over $\oM_{1,1}$.
\qed

\vspace{15pt}

Using the rationality of $\oM_{1,10}$, Proposition \ref{eeee}
can be easily strengthened to show that
the locus of irreducible $(g-1)$-nodal curves in $\oM_{g,9}$ 
 is rational.
 In particular, every such curve defines a tautological cycle 
$$[\widehat{C}, p_1,\ldots,p_9] \in R_0(\oM_{g,9})\, .$$

\subsection{Abelian surfaces}
Let $A$ be a nonsingular projective Abelian surface.
An irreducible nonsingular curve 
$$C\subset A$$
 is expected to move in a linear series $|L|$  of
dimension $g-2$. We therefore expect to find curves
$\widehat{C}\in |L|$ with $g-2$ nodes. 
Unfortunately the strategy that we have outlined in the case of
Enriques surfaces fails here! The locus of 
 irreducible $(g-2)$-nodal curves in $\oM_{g}$ 
 is {\em not} always rational. The irrationality
of the locus of $7$ nodal curves in $\oM_{9}$ 
was proven with Faber using the non-triviality (and
representation properties) of $H^{14,0}(\oM_{2,14})$.
A study of the Kodaira dimensions of 
the loci of curves with multiple nodes in many (other)
cases can be found in \cite{Sch}.

Nevertheless, an affirmative answer to the following
question appears likely.

\begin{question}\label{abb}
Does every irreducible nonsingular curve $C\subset A$ of genus $g$
determine a tautological $0$-cycle $[C]\in A_0(\oM_g)$?
\end{question}

Another approach to Question \ref{abb} is to use  
curves on $K3$ surfaces via the Kummer construction. 
Using the involution $$\iota: A \to A\, ,\ \ a \mapsto -a\, ,$$ we obtain a $K3$ surface $S$ by resolving the singular points of the quotient $A/\iota$. If $C$ does not meet any of these $16$ points (which are the fixed-points of $\iota$), the corresponding rational map
\[A \to A/\iota \dashrightarrow S\]
is defined around $C \subset A$ and sends $C$ to a curve $C' \subset S$. The map $C \to C'$ is either a double cover (in which case it must be \'etale with $C'$ smooth) or birational. In the first case, $[C']$ is tautological by Theorem \ref{curvek3} which may help in proving that $[C]$ is tautological. In the second case, the curve $C$ is the normalization of $C'$, and we would require a variant of Theorem \ref{curvek3}  to show that, under suitable conditions, the normalization of an irreducible, nodal curve in a $K3$ surface is tautological.


\subsection{Surfaces of general type}
Let $S$ bs a nonsingular projective surface of general type.
A curve $C\subset S$ is {\em canonical} if
$$\omega_S \stackrel{\sim}{=} \mathcal{O}_S(C)\, .$$
The most basic question which can be asked is the following.

\begin{question}\label{abb2}
Does every irreducible nonsingular {\em canonical} 
curve $C\subset S$ of genus $g$
determine a tautological $0$-cycle $[C]\in A_0(\oM_g)$?
\end{question}

\noindent For surfaces $S$ arising as complete intersections
in projective space, the answer to Question \ref{abb2} is yes (since
complete intersection curves are easily seen to determine
tautological $0$-cycles by degenerating their defining equations
to products of linear factors). 
However, even for surfaces of general type arising as
double covers of $\mathbb{P}^2$, the issue does not
appear trivial (even though the canonical curves there are
realized as concrete double covers of plane curves).
In fact, Question \ref{abb2} is completely open in almost all
cases.

\section{Cyclic covers} \label{Sect:cycliccovers}
If a nonsingular projective complex curve $C$ admits a
Hurwitz covering of $\mathbb{P}^1$ ramified over only
3 points of $\mathbb{P}^1$, then $C$ can be defined
over $\overline{\mathbb Q}$ by Belyi's Theorem. Speculation \ref{ddd3}, for
$n=0$, then suggests that the moduli point of $C$ is tautological.
The following result proves a special case for cyclic 
covers.{\footnote{Following the notation of \cite{admcycles}, Theorem \ref{Thm:cycliccovers}
shows that the $0$-cycle
\[[\overline{\mathcal{H}}_{g,\, \mathbb{Z}/k\mathbb{Z},\,  (a,b,c)}] \in A_0(\oM_{g,n}) \]
is tautological for $a,b,c \in \mathbb{Z}/k\mathbb{Z}$ where
at least one of $a,b,c$ is coprime to $k$.}}

\begin{theorem} \label{Thm:cycliccovers}
Let $C$ be a nonsingular projective curve of genus $g$ admitting a cyclic cover 
$$\varphi: C \to \mathbb{P}^1$$
ramified over exactly three points of $\mathbb{P}^1$ and
with total ramification over at least one of them. 
Let $p_1, \ldots, p_n \in C$ be the ramification points of $\varphi$ (in some order). Then, the $0$-cycle
$$[C,p_1, \ldots, p_n] \in A_0(\oM_{g,n})$$ is tautological.
\end{theorem}

\proof
The basic idea is that a cyclic cover of $\mathbb{P}^1$ can (essentially) be cut out by a single equation in a projective bundle over $\mathbb{P}^1$. Indeed, after a
change of coordinates, we can assume that the branch points of $\varphi$ are given by $$0,1,2 \in \mathbb{P}^1\, . $$
Let $k$ be the degree of $\varphi$, and let $a,b,c \in \mathbb{Z}/k\mathbb{Z}$ be the monodromies of $\varphi$ at the branch points $0,1,2$ satisfying
$$a+b+c=0 \in \mathbb{Z}/k\mathbb{Z}\, . $$ Assume that the total ramification occurs over $0$. Then $a$ is coprime to $k$, and, by applying an automorphism of $\mathbb{Z}/k\mathbb{Z}$, we may assume $a=1$. We can then choose representatives $$b,c \in \{1, \ldots, k-1\}$$ such that $a+b+c=k$. 

With these choices in place, we see that (birationally) the curve $C$ is cut out in the projectivization of the line bundle $\mathcal{O}_{\mathbb{P}^1}(1)$
over ${\mathbb{P}^1}$ by the equation
\begin{equation} \label{eqn:cycliccover}
    y^k = x \cdot (x-1)^b \cdot (x-2)^c,
\end{equation}
where $x$ is a coordinate on the base $\mathbb{P}^1$.
We view 
the right hand side of (\ref{eqn:cycliccover}) as a section of $$\mathcal{O}_{\mathbb{P}^1}(1+b+c)=\mathcal{O}_{\mathbb{P}^1}(k)$$ 
where $y$ is the coordinate on (the total space of) the line bundle $\mathcal{O}_{\mathbb{P}^1}(1)$ over $\mathbb{P}^1$.

We say that $C$ is cut out birationally since, for $b,c \neq 1$, the above curve will have singularities at $$(x,y)=(1,0),(2,0)\, .$$
The
singularities
can be resolved by performing a
specific sequence of iterated blowups (as will be explained
in the next paragraph).
After finitely many steps, we will 
obtain $C$ sitting inside a blowup $S$ of $$\mathbb{P}=\mathbb{P}(\mathcal{O}_{\mathbb{P}^1}(1) \oplus \mathcal{O}_{\mathbb{P}^1})\,,$$
which is a nonsingular rational surface. 
In order to conclude by applying Theorem \ref{curveratsur},
we will have to check that
$$\int_{[C]} c_1(S)>0$$ 
holds and that the number $n$ of ramification points of $\varphi$ is at most equal to $\mathrm{vdim}\,\oM_{g}(S,[C])$.

The original curve $C_0$ in $\mathbb{P}$ is easily seen to be of class $\beta=k c_1(\mathcal{O}_{\mathbb{P}}(1))$, and 
we have $$\int_{\beta} c_1(\mathbb{P})=3k\, .$$ 
If $b>1$, then $C_0$ has a singularity of
multiplicity $b$ at $(x,y)=(1,0)$. 
For the coordinate $z=y/(x-1)$ on the blowup of $\mathbb{P}$ at $(1,0)$, the strict transform of $C_0$ is locally cut out by $z^{k-b}=(x-1)^b$. The relevant intersection number 
$$\int_{\beta-b E_1} c_1(\mathrm{Bl}_{(1,0)} \mathbb{P})=3k-b$$
has exactly decreased by the multiplicity $b$ of $C_0$ at $(1,0)$.

We can continue the process of blowing-up the singular point and taking the strict transform.
After $j$ steps, the curve still has a local equation of the form $z_1^{e_j} = z_2^{f_j}$. We started with $(e_0,f_0)=(k,b)$ and obtained $$(e_1,f_1)=(k-b,b)$$ in the first step. In general, the pairs $(e_j,f_j)$ are then obtained by performing a Euclidean algorithm starting from $(k,b)$. The multiplicity of the singular point after the $j$th step is exactly $\min(e_j,f_j)$. 
The process terminates after finitely many steps (when the minimum of $e_j,f_j$ is either $0$ or $1$). Then, the local equation is $z^g=1$ or $z^g=z'$, which is nonsingular.

Denote by $\mathrm{ms}(e,f)$ the sum of the multiplicities of the singular points that occur in the desingularization of $z_1^e=z_2^f$ in the above manner. The function is uniquely determined by the axioms
\begin{itemize}
    \item $\mathrm{ms}(e,f)=\mathrm{ms}(f,e)$,
    \item $\mathrm{ms}(e,0)=\mathrm{ms}(e,1)=0$,
    \item $\mathrm{ms}(e,f)=f + \mathrm{ms}(e-f,f)$, for $e\geq f$.
\end{itemize}
By the above analysis, the curve $C\subset S$ obtained by desingularizing $C_0\subset \mathbb{P}$ satisfies
\[\int_{[C]} c_1(S) = \int_{\beta} c_1(\mathbb{P}) - \mathrm{ms}(k,b) - \mathrm{ms}(k,c) = 3k - \mathrm{ms}(k,b) - \mathrm{ms}(k,c).\]
In order to show positivity, 
we must bound $\mathrm{ms}(e,f)$ from above. By induction, 
for $(e,f) \neq (1,1)$, we obtain:
\begin{equation} \label{eqn:boundms}
    \mathrm{ms}(e,f) \leq e + f - R(e,f)\text{, with }R(e,f)=
    \begin{cases}
    \gcd(e,f) &\text{ if }\gcd(e,f)>1\, ,\\
    3 &\text{ otherwise}\, .
    \end{cases}
\end{equation}
Then, we have
\begin{equation}
    \int_{[C]} c_1(S) \geq 3k - k-b-k-c + R(k,b)+R(k,c)= 1 + R(k,b)+R(k,c) \geq 1\, .
\end{equation}
For the virtual dimension we obtain 
\[\mathrm{vdim}\,\oM_{g}(S,[C])=\int_{[C]} c_1(S)+g-1 \geq g + R(k,b)+R(k,c)\, . \]
On the other hand, the number of ramification points equals $$n=1+\gcd(k,b)+\gcd(k,c)\,$$ so we have
\[\mathrm{vdim} - n \geq g-1 + \underbrace{(R(k,b)-\gcd(k,b))}_{\geq 0} + \underbrace{(R(k,c)-\gcd(k,c))}_{\geq 0} \geq g-1\, ,\]
which we can assume to be nonnegative. We have thus verified the assumptions of Theorem \ref{curveratsur}.
\qed
\vspace{10pt}

Without the assumption of total ramification over one of the three points, the proof technique above no longer works. Indeed, for $k=30$ and $$(a,b,c)=(2,3,25)\,,$$ a desingularization procedure over $x=0,1,2$ as in the above proof would result in a curve $C$ in $S$ satisfying $$\int_{[C]} c_1(S)=-20\, ,$$ which cannot be remedied by applying an automorphism of $\mathbb{Z}/30\mathbb{Z}$. 
Nevertheless, we expect Theorem \ref{Thm:cycliccovers} to hold without the assumption of total ramification and even without the assumption of the cover being cyclic.

\section{Summing to tautological cycles} \label{sT}
\subsection{Existence}
As the examples $\oM_{1,n\geq 11}$  show, 
the Chow group of $0$-cycles on $\oM_{g,n}$ can be infinite dimensional
over $\mathbb{Q}$.  
The general point of $\oM_{g,n}$ may not determine
 a tautological $0$-cycle. However, by adding points 
(with the number of points uniformly bounded in terms of $g,n$),
we can arrive at a tautological $0$-cycle. 
For technical reasons, we formulate the  result
 for the coarse moduli space $\overline{M}_{g,n}$.

\begin{proposition} \label{siblings}
 Given $g,n$ with $2g-2+n>0$, there exists an integer $T=T(g,n) \geq 1$
satisfying the following property: 
for any point 
$$Q_1=(C,p_1, \ldots, p_n) \in \overline{M}_{g,n}\, ,$$ 
we can find $Q_2, \ldots, Q_{T} \in \overline{M}_{g,n}$ such that
\[[Q_1] + \ldots + [Q_T] \in A_0(\overline{M}_{g,n})\]
is tautological. 
\end{proposition}

\begin{proof}[Proof \emph{(suggested by A. Kresch)}]
By standard arguments using the results of Section \ref{bracc},
we may take $Q=Q_1$ to be a general point of
$\overline{M}_{g,n}$.
We then choose a very ample divisor class 
$$H \in A^1(\overline{M}_{g,n})\, .$$ Since $Q$ is a
nonsingular point of $\overline{M}_{g,n}$, general hyperplane sections 
$$H_1, \ldots, H_{3g-3+n} \in |H|$$ 
through $Q$ will
intersect transversely
in a union of reduced points 
$$\alpha=[Q_1]+ \cdots +[Q_T]\, ,$$
with $T=\mathrm{deg}(\overline{M}_{g,n},H)$. 
On the other hand, since all divisor classes on $\overline{M}_{g,n}$ are tautological, the class $\alpha$ is also tautological.
\end{proof}

\vspace{10pt}

\begin{remark} \label{Rmk:coarsemod}
{\em  Since the push-forward along the basic map
 $$\oM_{g,n} \to \overline{M}_{g,n}$$ 
is an isomorphism of $\mathbb{Q}$-Chow groups, we
 can derive a version of Proposition \ref{siblings}
 with $\overline{M}_{g,n}$ replaced by $\oM_{g,n}$.  
However,  $T(g,n)$ for $\overline{M}_{g,n}$,  
may differ from the corresponding number for $\oM_{g,n}$: 
if $Q_i \in \oM_{g,n}$ has nontrivial automorphisms, then the cycle $[Q_i] \in A_0(\overline{M}_{g,n})$ corresponds to the cycle 
$$|\mathrm{Aut}(Q_i)| \cdot [Q_i] \in A_0(\oM_{g,n})\, .$$ }
\end{remark}

\subsection{Minimality}
We denote by $T(g,n)$ the minimal integer having the property described 
in Proposition \ref{siblings}. 
The proof of Proposition \ref{siblings}
used the degree of $\oM_{g,n}$, but there are
several other geometric approaches to bounding $T(g,n)$. 
For example, we could use instead
the Hurwitz cycle results of \cite{faberpandharipande}. 
After fixing a degree $d \geq 1$, points $q_1, \ldots, q_b \in \mathbb{P}^1$, and partitions $\lambda_1, \ldots, \lambda_b$ of $d$, the sum of all points $[(C,(p_i)_i)]$ satisfying
\begin{itemize}
    \item there exists a degree $d$ map $C \to \mathbb{P}^1$ 
    with ramification profile $\lambda_j$ over $q_j \in \mathbb{P}^1$,
    \item with $(p_i)_i$ the set of preimages of the points $q_1, \ldots, q_b$
\end{itemize}
is tautological by \cite{faberpandharipande}. 
Since every genus $g$ curve $C$ admits \emph{some} map $C \to \mathbb{P}^1$, the result above implies that adding to $[C] \in A_0(\oM_{g})$ all cycles $[C']$ for curves $C' \to \mathbb{P}^1$ with the same branch points and ramification profiles as $C \to \mathbb{P}^1$ gives a tautological class. Hence, we  bound $T(g,0)$ in terms of a suitable Hurwitz number. A similar strategy works for any $n$ by including the markings $p_1, \ldots, p_n \in C$ among the ramification data of $C \to \mathbb{P}^1$.

However, these
approaches will likely not yield optimal bounds.
In all the cases listed in Figure \ref{Fig:ratconnMgn}, the space $\oM_{g,n}$ is rationally connected, so 
$$T(g,n)=1\, ,$$
 which is far below the bounds.

A different perspective on the question is to study 
 the behavior of $T(g,n)$ for fixed $g$ as $n\rightarrow 
\infty$.
The following result shows
 that the asymptotic growth in $n$ is {\em at most} linear.

\begin{proposition}\label{lineargrowth}
 Let $(g,n)$ satisfy $2g-2+n >0$. Then,
 \begin{equation} \label{eqn:Srecursion}
  T(g,n+m) \leq (gm+1)\cdot T(g,n)\, 
 \end{equation}
for all $m\geq 0$.
\end{proposition}
\proof
 The natural forgetful map 
 \[\nu: \oM_{g,n+m} \to \oM_{g,n}\]
 has a section $\sigma$ defined by the
following construction: 
 $\sigma((C,p_1, \ldots, p_n))$ is the curve obtained by gluing a chain of rational curves containing the markings $p_{n}, \ldots, p_{n+m}$ at the previous position of $p_n \in C$. 
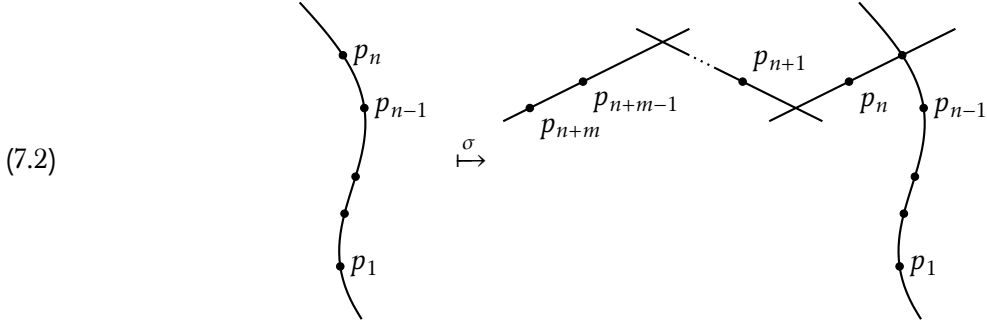
 \begin{equation}
    \vcenter{\hbox{
    \begin{tikzpicture}[scale=0.7]

\draw[domain=-3:3,smooth,variable=\y,thick]  plot ({0.7*sin(50*\y+25)-0.3*\y},{\y});
\filldraw ({0.7*sin(50*(1)+25)-0.3*(1)},{1}) circle (2pt) node[right]{$p_{n-1}$};
\filldraw ({0.7*sin(50*(-1)+25)-0.3*(-1)},{-1}) circle (2pt);
\filldraw ({0.7*sin(50*(-0.3)+25)-0.3*(-0.3)},{-0.3}) circle (2pt);
\filldraw ({0.7*sin(50*(-2)+25)-0.3*(-2)},{-2}) circle (2pt) node[right]{$p_{1}$};
\filldraw ({0.7*sin(50*(2)+25)-0.3*(2)},{2}) circle (2pt) node[right]{$p_{n}$};
\end{tikzpicture} }} \xmapsto{\sigma}
   \vcenter{\hbox{
    \begin{tikzpicture}[scale=0.7]

\draw[domain=-3:3,smooth,variable=\y,thick]  plot ({0.7*sin(50*\y+25)-0.3*\y},{\y});
\filldraw ({0.7*sin(50*(1)+25)-0.3*(1)},{1}) circle (2pt) node[right]{$p_{n-1}$};
\filldraw ({0.7*sin(50*(-1)+25)-0.3*(-1)},{-1}) circle (2pt);
\filldraw ({0.7*sin(50*(-0.3)+25)-0.3*(-0.3)},{-0.3}) circle (2pt);
\filldraw ({0.7*sin(50*(-2)+25)-0.3*(-2)},{-2}) circle (2pt) node[right]{$p_{1}$};
\filldraw ({0.7*sin(50*(2)+25)-0.3*(2)},{2}) circle (2pt);
\draw[thick] ({0.7*sin(50*(2)+25)-0.3*(2)+1},{2+0.5}) -- ({0.7*sin(50*(2)+25)-0.3*(2)-1},{2-0.5}) node[below right]{$p_n$}-- ({0.7*sin(50*(2)+25)-0.3*(2)-2.5},{2-1.25});
\filldraw ({0.7*sin(50*(2)+25)-0.3*(2)-1},{2-0.5}) circle (2pt);
\draw[thick] ({0.7*sin(50*(2)+25)-0.3*(2)-1.5},{2-1.25}) -- ({0.7*sin(50*(2)+25)-0.3*(2)-3.5},{2-0.25});
\draw[thick,dotted] ({0.7*sin(50*(2)+25)-0.3*(2)-3.5},{2-0.25}) -- ({0.7*sin(50*(2)+25)-0.3*(2)-4},{2-0});
\draw[thick] ({0.7*sin(50*(2)+25)-0.3*(2)-4},{2-0}) -- ({0.7*sin(50*(2)+25)-0.3*(2)-5},{2+0.5});
\filldraw ({0.7*sin(50*(2)+25)-0.3*(2)-3},{2-0.5}) circle (2pt)node[above right]{$p_{n+1}$};
\draw[thick] ({0.7*sin(50*(2)+25)-0.3*(2)-4},{2+0.5}) -- ({0.7*sin(50*(2)+25)-0.3*(2)-7.5},{2-1.25});
\filldraw ({0.7*sin(50*(2)+25)-0.3*(2)-6},{2-0.5}) circle (2pt) node[below right]{$p_{n+m-1}$};
\filldraw ({0.7*sin(50*(2)+25)-0.3*(2)-7},{2-1}) circle (2pt) node[below right]{$p_{n+m}$};
\end{tikzpicture} }}
 \end{equation}
The section $\sigma$
is a composition of suitable boundary gluing maps, so the push-forward of a tautological cycle via $\sigma$ is tautological.

Let $Q\in \cM_{g,n}$ be a moduli point with 
a {\em nonsingular} domain curve $C$. We claim:
{\em for every $Q_1 \in \nu^{-1}(Q)$, 
 there exist  $Q_2, \ldots, Q_{gm+1} \in \nu^{-1}(Q)$ satisfying}
 \[[Q_1] + \cdots + [Q_{gm+1}] = (gm+1)[\sigma(Q)] \in A_0(\nu^{-1}(Q))\, .\]
Assuming the above claim, we can easily finish the
proof. 

Let $Q_1 \in \cM_{g,n+m}$ with $Q=\nu(Q)$ and $Q_2, \ldots, Q_{gm+1}$ as 
in the above claim. 
By the definition of $T(g,n)$, we can find 
$$P_1=Q, P_2, \ldots, P_{T(g,n)} \in \overline{M}_{g,n}$$ for which
$$[P_1] + \cdots + [P_{T(g,n)}]\in R_0(\overline{M}_{g,n})\, .$$ 
We then obtain
 \begin{align*}
  [Q_1] + \cdots + [Q_{gm+1}] + \sum_{i=2}^{T(g,n)}
(gm+1) [\sigma(P_i)] &= (gm+1)[\sigma(Q)] + \sum_{i=2}^{T(g,n)} 
(gm+1) [\sigma(P_i)] \\
  &= (gm+1) \sigma_*\left([P_1] + \cdots + [P_{T(g,n)}] \right) \in R_0(\overline{M}_{g,n+m})\, .
 \end{align*}
 Hence, $T(g,n+m) \leq (gm+1)\cdot T(g,n)$.
 
We now prove the required claim.
For $$Q=(C,p_1, \ldots, p_n) \in \cM_{g,n}\, ,$$ 
the fibre $\nu^{-1}(Q)$ is isomorphic to a blow-up of the product $C^{m}$. 
Since the natural map 
$$\nu^{-1}(Q) \to C^{m}$$
 is a birational morphism between nonsingular varieties, 
we have an induced isomorphism 
$$A_0(\nu^{-1}(Q)) \to A_0(C^{m})$$
 by \cite[Example 16.1.11]{fulton}. We can therefore
 verify the claim on $C^{m}$ instead of $\nu^{-1}(Q)$. 
The image of $\sigma(Q)$  in $C^m$
is exactly the point 
$$(p_n, \ldots, p_n) \in C^{m}\, .$$
 
By Riemann-Roch, every line bundle on $C$ of degree at least $g$ is effective. 
In other words, any divisor of degree at least $g$ can be written as a sum of points on $C$.
Assume we are given 
$$Q_1=(x_1, \ldots, x_{m}) \in C^{m}\, .$$ 
 Then, there exist points $x_{1,1}, \ldots, x_{1,g} \in C$ satisfying
 \[[x_{1,1}] + \cdots + [x_{1,g}] = (g+1)[p_n] - [x_{1}] \in A_0(C)\, .\]
 Let $Q_{i+1} = (x_{1,i},x_2, \ldots, x_{m}) \in C^{m}$ for $i=1, \ldots, g$.
 We have 
 \[[Q_1] + \cdots + [Q_{g+1}] = (g+1) [(p_n,x_2, \ldots, x_{m})] \in 
A_0(C^{m})\, .\]
 For the next step, there exist points
$x_{2,1}, \ldots, x_{2,g} \in C$ satisfying
 \[[x_{2,1}] + \cdots + [x_{2,g}] = (2g+1)[p_n] - (g+1)[x_{2}] 
\in A_0(C).\]
 Let $Q_{g+1+i}=(p_n,x_{2,i},x_3, \ldots, x_{n'})$ for $i=1, \ldots, g$.
 We have
 \[[Q_1] + \cdots + [Q_{2g+1}] = (2g+1) [(p_n,p_n,x_3, \ldots, x_{n'})] \in 
A_0(C^{m}).\]
After iterating the above procedure, we find points
$Q_1, \ldots, Q_{gm+1}$ satisfying
 \[[Q_1] +  \cdots + [Q_{gm+1}] = (gm+1) [(p_n, \ldots, p_n)] 
\in A_0(C^{m})\]
 as desired.
\qed

\begin{question}
Does $T(g,n)$ really grow linearly as $n\rightarrow \infty$?
\end{question}

By results{\footnote{We thank Qizheng Yin for pointing
out the connection.}} 
of Voisin (see Theorem 1.4 of \cite{VAV}), the
analogous $T$ number of an abelian variety $A$ is at least $\dim(A)+1$.
The linear growth there perhaps also suggests 
a linear lower bound for $T(g,n)$ as $n\rightarrow \infty$.


\ifx\undefined\bysame
\newcommand{\bysame}{\leavevmode\hbox to3em{\hrulefill}\,}
\fi

\end{document}